\documentclass{amsart}

\usepackage{amsmath,amsbsy,amsfonts,amsopn,amstext}
\usepackage{euscript,amssymb,amsbsy,amsfonts,amsthm,latexsym,amsopn,amstext,amsxtra,euscript,amscd}
\usepackage{graphicx, verbatim}

\usepackage{longtable}

\usepackage{xy} 
\input xy \xyoption{all}

\usepackage{graphicx}

\usepackage{hyperref}

\hypersetup{
  colorlinks   = true, 
  urlcolor     = blue, 
  linkcolor    = blue, 
  citecolor   = red 
}

\newtheorem{prop}{Proposition} 
\newtheorem{lem}{Lemma} 
 
\newtheorem{defn}{Definition} 

\newtheorem{exa}{Example} 

\newtheorem{rem}{Remark}
\newtheorem{cor}{Corollary}

\DeclareMathOperator\Sl{SL}

\newcommand\Z{\mathbb Z}

\newcommand\R{\mathbb R}
\newcommand\C{\mathbb C}

\def\P{\mathbb P}

\def\H{\mathcal H}

\newcommand\F{\mathcal F}            
\newcommand\J{\mathcal J}
\def\O{\mathcal O}

\def\<{\langle}

\newcommand\z{\xi}  

\def\a{\alpha}

\usepackage{amsaddr}

\usepackage[msc-links]{amsrefs}

\usepackage{tikz}
\usepackage{tikz,ifthen}
\usepackage{calc,rotating}
\usetikzlibrary{calc,fadings, intersections, decorations.pathreplacing}
\usetikzlibrary{arrows,positioning, snakes}
\usetikzlibrary{fit}
\usetikzlibrary{matrix}
\usepackage{tikz-3dplot}

\title{Reduction of binary forms via the hyperbolic center of mass}

\author{A. Elezi}
\address{Department of Mathematics and Statistics \\ American University \\ Washington, DC, 20016. \\
Email: \, aelezi@american.edu}

\author{T. Shaska}
\address{Department of Mathematics and Statistics \\ Oakland  University \\ Rochester, MI, 48309. \\
Email:  \, shaska@oakland.edu}

\date{}                                           

\begin{document}
\maketitle


\begin{abstract}
In this paper we provide an alternative reduction theory for real, binary forms with no real roots. Our approach is completely geometric, making use of the notion of hyperbolic center of mass in the upper half-plane. It appears that our model compares favorably with existing reduction theories, at least in certain aspects related to the field of definition. Various tools and features of hyperbolic geometry that are interesting in themselves, but also relevant for our and various other reduction theories papers (\cite{julia} and \cite{SC}), are also treated in detail and in a self-contained way here.
\end{abstract}
\section{Introduction}\label{red-bin-forms}

During the XIX-century, the mathematical community invested much efforts in developing a reduction theory of binary forms similar to that of quadratic forms, especially since invariant theory was at the forefront of mathematics.  The idea of reduction on a set $A$ with a right $\Sl_2(\Z)$-action is to associate to any element $a\in A$ a covariant point $\z(a)$ in the upper half-plane $\H_2$, i.e to construct an $\Sl_2(\Z)$-equivariant map $\z:A\rightarrow \H_2$. The modular group $\Sl_2(\Z)$ acts on binary forms $F(X,Z)$ via a linear change of variables and on the upper half-plane via M{\"o}bius transformations. A practical motivation for the reduction in this setting is: given a real binary form, can we find an $\Sl_2(\Z)$-equivalent with minimal coefficients? This question has a positive answer for quadratics but it is still not very well understood for higher degree forms.  
  
In $1917$, G. Julia introduced in his thesis \cite{julia} a reduction theory for binary forms with real coefficients, although explicit and complete answers were provided only in degrees three and four. To every binary form $F(X,Z)$ with real coefficients, Julia associated a positive definite quadratic $\J_F$ called the \textit{Julia quadratic}. The set of positive definite quadratics parametrizes the upper half-plane via one of its roots. Hence, there exists a well defined map, called \textit{the zero map}, from the set of real binary forms to the upper half plane. This map is $\Sl_2(\Z)$-equivariant.  A binary form is called \textit{reduced} if its image via the zero map is in the fundamental domain $\F$ of $\Sl_2(\Z)$.   

In $1999$,  Cremona \cite{cremona-red} used the ideas of Julia to explore the reduction for cubic and quartic binary forms.  He showed that the coefficients of the Julia quadratic of a cubic form $F$ are polynomial values of of the coefficients of $F$ and this does not happen for higher degree forms. 

In \cite{SC} Cremona and Stoll developed a reduction theory in a unified setting for binary forms with real or complex coefficients. Generalizing Julia's work, a positive definite Hermitian quadratic $\J_F$ is associated to every binary complex form $F(X,Z)$ of degree $n\geq 2$. In his thesis Julia showed the existence of $\J_F$ and proved that it is a \textit{covariant} of the binary form. The uniqueness of $\J_F$ was shown in  \cite{SC}. Positive definite Hermitian forms parametrize the upper half-space $\H_3$.  This upper half-plane $\H_2$ is contained in the upper half-space $\H_3$ as a vertical cross section (see the following section). When the form $F(X,Z)$ has real coefficients, compatibility with complex conjugation (see the comments after Corollary \ref{tangent vectors}) forces $\J_F$ to live in $\H_2$. It is in this sense that working in $\H_3$ unifies the theory of real and complex binary forms. A degree $n$ complex binary form $F(X,Z)$ is called \textit{reduced} when its zero map value $\z(\J_F)$ is in the fundamental domain of the action of the modular group $\Sl_2(\C)$ on $\H_3$. 

In the works cited above, the term \textit{reduced binary form} means reduced in the $\Sl_2 (\Z)$ orbit.  It is expected that the reduced forms have smallest size coefficients in such orbit.  In \cite{b-sh} the concept of height was defined for forms defined over any ring of integers $\O_K$, for any number field $K$, and the notion of \textit{minimal absolute height} was introduced. In \cite{beshaj}, the author suggests an algorithm for determining the minimal absolute height for binary forms.  Continuing with this idea, a database of binary sextics of minimal absolute height $\mathfrak h \leq 10$ together with many computational aspects of binary sextics are included in \cite{data}. 

The primary goal of this paper is (a) to provide in a self-contained way all the details and the background of the geometry behind the previously mentioned binary form reductions and (b) to introduce an alternative reduction based on the pure geometric notion of hyperbolic center of mass in $\H_2$. For cubics and quartics, in \cite{julia} Julia uses geometric constructions to establish the barycentric coordinates $t_1, \dots , t_n$ of the zero map in the hyperbolic convex hull of the roots of $F$. In \cite{SC} a slightly different positive definite Hermitian form is used for the reduction of binary complex forms.  Our reduction is based solely on geometric ideas. We will discuss whether such reduction has any benefits compared to the previous ones.

In section $2$, we describe in detail the reduction relevant features of the hyperbolic geometry of the upper half plane $\H_2$ and upper half-space $\H_3$. These spaces are shown to parameterize respectively the positive definite quadratics and the positive definite Hermitian forms. We prove that these parameterizations respect the corresponding structures: for any $n$ points $w_1, \dots , w_n \in \H_3$, the hyperbolic convex hull of these points parametrizes the positive linear combinations $\sum_{i=1}^n \lambda_i H_{w_i} (x)$, where $H_{w_i} (x)$ is the positive definite Hermitian form corresponding to $w_i$. 

In section $3$, we summarize the reduction theory developed in \cite{julia} and \cite{SC}. We focus especially on the geometrical aspects of the zero map and the reduction, as these are of special interest to us.

In section $4$ we define the \textit{hyperbolic center of mass} of a collection $\{w_1,...w_n\}\subset \H_2$ as the unique point $x$ inside their hyperbolic convex hull which minimizes 
\break $\sum_{i=1}^n \cosh (d_H ({\mathbf x}, w_i)$ (here $d_H$ is the hyperbolic distance). To each real binary form $F(X,Z)$ with no real roots, our alternative zero map associates the hyperbolic center of mass of its roots. We show that this map is $SL_2 (\R)$ equivariant, hence it defines a new reduction algorithm. We note that our zero map is different from the one used in \cite{julia} or \cite{SC}

It does seem that computationally this reduction does produce binary forms of smaller height as in the case of reduction suggested by Julia or Cremona/Stoll. Naturally, one would like to determine how "far" this  zero map is from the zero map suggested by Julia or whether one can get examples that such different reductions gives different results.  In section 4 we perform some computations with binary forms with no real roots (see also \cite{beshaj} for totally complex forms). 


\section{The hyperbolic geometry of positive definite binary forms} 

In this section we present some features of hyperbolic geometry that are not only relevant for the reduction theory of binary forms. but are also interesting on their own. We also establish a correspondence between hyperbolic spaces and positive definite quadratic forms.

\subsection{The hyperbolic plane $\H_2$} The upperhalf-plane equipped with the Riemanian metric $$ds^2=\frac{dx^2+dy^2}{y^2}$$ is one of the models of the two dimensional hyperbolic space. It is denoted by $\H_2$.  The geodesics of the Riemaniann manifold $\H_2$, i.e the hyperbolic equivalents of Euclidean straight lines, are either semicircles $C_{a,b}$ with diameter from $A(a,0)$ to $B(b,0)$ on the real axis, or the vertical rays $C_a$ with origin at $x=a$. In the standard literature, the points $A(a,0), B(b,0)$ are called \textit{the ideal points} of the geodesic $C_{a,b}$, likewise $A(a,0)$ and $\infty$ are the ideal points of $C_a$. They live in the boundary of $\H_2$ as it can be seen from Fig.~\ref{fig1}.

\begin{figure}[htbp] 
   \centering
\begin{tikzpicture}[/.style={blue!30,very thin},scale=0.65]
    \draw [help lines] (-2, 0)  (16, 5);
    \draw[->, thick] (-2, 0) -- (16, 0) node[below]{\footnotesize $x$};
    \draw[->, thick] (0, 0) -- (0, 5) node[right]{\footnotesize $y$};

\draw[,  thick ] (7, 0) arc (0:180:3cm);

    \draw[-, ,thick] (11, 0) -- (11, 5);

    \node at (1, -0.5) {$A$};
        \node at (7, -0.5) {$B$};
            \node at (11, -0.5) {$A$};
\end{tikzpicture}

   \caption{Geodesics and their ideal points}
   \label{fig1}
\end{figure}
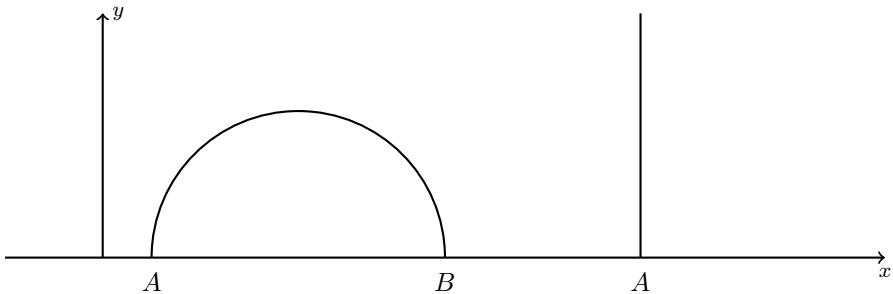

The hyperbolic distance between two points $z=x+{\bf i}y$ and $w=u+{\bf i}v$ is computed as follows. Let $z_{\infty}, w_{\infty}$ be the ideal points of the geodesic through $z,w$, where $z_{\infty}$ is the one closer to $z$; see Fig.~\ref{fig2}.  
\begin{figure}[htbp] 
   \centering

\begin{tikzpicture}[/.style={blue!30,very thin},scale=0.6]
    \draw [help lines] (-2, 0) (16, 5);
    \draw[->, thick] (-2, 0) -- (16, 0) node[below]{\footnotesize $x$};
    \draw[->, thick] (0, 0) -- (0, 5) node[right]{\footnotesize $y$};

\draw[, thick ] (7, 0) arc (0:180:3cm);

    \draw[-, ,thick] (11, 0) -- (11, 5);

    \node at (11.5, 2) {$z$};
    \node at (11.5, 4) {$w$};
    
    \foreach \x/\y in {   4/3, 2/2.2, 11/2, 11/4}
    \filldraw[] (\x, \y) circle(3pt);

    \node at (2, 2.7) {$z$};
    \node at (4, 3.5) {$w$};
     
    \node at (1, -0.5) {$z_{\infty}$};
    \node at (7, -0.5) {$w_{\infty}$};

\end{tikzpicture}
   \caption{The hyperbolic distance between two points $z$ and $w$ with $\Re (z) \neq \Re(w) $ and $\Re (z) = \Re(w) $ }
   \label{fig2}
\end{figure}
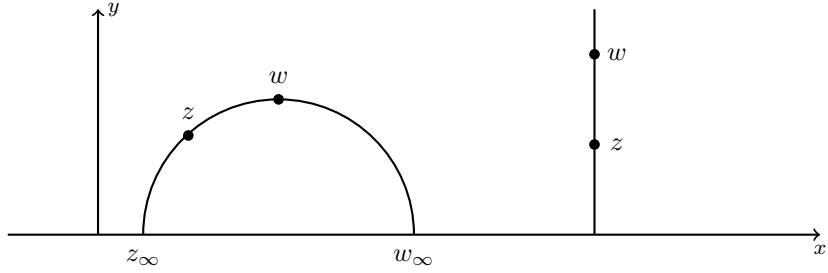

The hyperbolic distance is defined in terms of the cross-ratio or Euclidean distances 
\[ d_H(z,w)=\log[z,w,w_{\infty},z_{\infty}]=\log\left(\frac{z-w_{\infty}}{w-w_{\infty}}\frac{w-z_{\infty}}{z-z_{\infty}}\right)=\ln\left(\frac{|z-w_{\infty}|}{|w-w_{\infty}|}\frac{|w-z_{\infty}|}{|z-z_{\infty}|}\right).\]
Notice that for $x=u$ and $y<v$, the geodesic is the vertical ray $C_x$. In this case \break $z_{\infty}=(x,0), w_{\infty}=\infty$ and 
\[ d_H(z,w)=\ln \left(\frac{v}{y}\right).\] 
For $A(a,0)$ and $z=x+{\bf i}y\in \H_2$, define \[ d_H(A,z):=\ln\left(\frac{(x-a)^2+y^2}{y}\right).\]
\begin{figure}[hbp] 
   \centering
\begin{tikzpicture}[/.style={blue!30,very thin},scale=0.6]
    \draw [help lines] (-2, 0)  (9, 5);
    \draw[->, thick] (-2, 0) -- (9, 0) node[below]{\footnotesize $x$};
    \draw[->, thick] (0, 0) -- (0, 5) node[right]{\footnotesize $y$};

    \draw[-, , thick] (1, 0) -- (4, 4);
    \draw[-, , thick] (4, 0) -- (4, 4);

     \foreach \x/\y in {   1/0, 4/4, 4/0}
    \filldraw[] (\x, \y) circle(3pt);

    \node at (1, -0.5) {$A$};
        \node at (4.2, 4.3) {$z$};
\end{tikzpicture}

   \caption{The distance between $z\in \H_2$ and a boundary point $A$.}
   \label{fig3}
\end{figure}
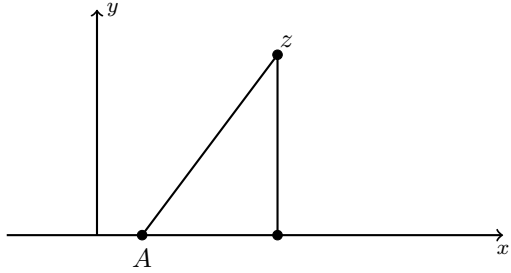

An additive property of this distance is claimed and used in \cite{SC}. To make the paper self-contained and for the benefit of the reader, we state and prove it below.

\begin{prop} 
Let $A$ be one of the ideal points of a geodesic that passes through $z=x+y{\bf i},w=u+v{\bf i}\in \H_2$. Then $d_H(z,w)=|d_H(A,z)-d_H(A,w)|$. 
\end{prop}

\begin{proof}
Assume first that $x\neq u$, i.e. the geodesic through $z$ and $w$ is a semicircle. Without loss of generality, assume that $A(0,0)$. 

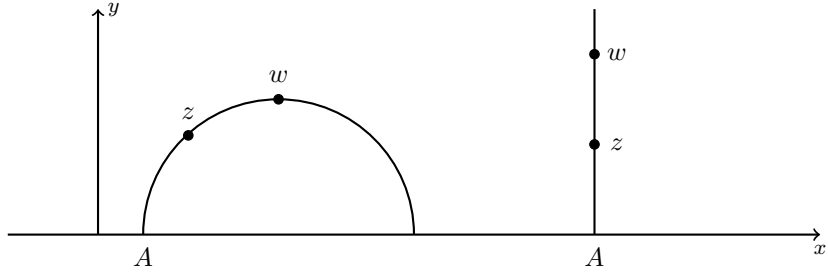
\begin{figure}[htbp] 
   \centering
\begin{tikzpicture}[/.style={blue!30,very thin},scale=0.6]
    \draw [help lines] (-2, 0)  (16, 5);
    \draw[->, thick] (-2, 0) -- (16, 0) node[below]{\footnotesize $x$};
    \draw[->, thick] (0, 0) -- (0, 5) node[right]{\footnotesize $y$};
    
\draw[,  thick ] (7, 0) arc (0:180:3cm);

    \draw[-, , thick] (11, 0) -- (11, 5);

    \node at (1, -0.5) {$A$};

            \node at (11, -0.5) {$A$};
            
    \node at (2, 2.7) {$z$};
    \node at (4, 3.5) {$w$};
     
    \node at (11.5, 2) {$z$};
    \node at (11.5, 4) {$w$};

     \foreach \x/\y in {   4/3, 2/2.2, 11/2, 11/4}
    \filldraw[] (\x, \y) circle(3pt);
        
\end{tikzpicture}
   \caption{The additive property of the boundary distance}
   \label{fig4}
\end{figure}

Let $(x-r)^2+y^2=r^2$ be the equation of the geodesic and $B(2r,0)$ the other ideal point. If $z(x,y),w(u,v)$, then $x^2+y^2=2rx,~u^2+v^2=2ru,~v^2=u(2r-u),~y^2=x(2r-x)$.  Now 
\[
\begin{split}
|d_H(A,z)-d_H(A,w)|   &   =\left|\ln\left(\frac{x^2+y^2}{y}\right)-\ln\left(\frac{u^2+v^2}{v}\right)\right | \\
& =\left|\ln\left(\frac{2rx}{y}\right)-\ln\left(\frac{2ru}{v}\right)\right |=\left |\ln\frac{xv}{yu}\right |
\end{split}
\] 
On the other hand, 
\[
\begin{split}
d_H(z,w) & =\left|\ln\left(\frac{|z||w-B|}{|w||z-B|}\right)\right|=\left|\ln\left(\frac{\sqrt{x^2+y^2}\sqrt{(2r-u)^2+v^2}}{\sqrt{(2r-x)^2+y^2}\sqrt{u^2+v^2}}\right)\right|\\
& =\left|\ln\left(\frac{\sqrt{2rx}\sqrt{4r^2-2ru}}{\sqrt{4r^2-2rx}\sqrt{2ru}}\right)\right|=\left |\ln\sqrt{\frac{x(2r-u)}{u(2r-x)}}\right|=\left |\ln\sqrt{\frac{x^2v^2}{y^2u^2}}\right|=\left |\ln \frac{xv}{yu}\right|.
\end{split}
\] 
 When $x=u$ the geodesic through $z,w$ is the ray $C_x$ with an ideal point at $A(x,0)$. 
Then, $d_H(A,z)=\ln y$ and $d_H(A,w)=\ln v$.  Hence, 
\[ d_H(z,w)=\left | \ln\frac{v}{y}\right |=|\ln v-\ln y|=|d_H(A,w)-d_H(A,z)|.\]
This completes the proof. 
\end{proof}

The group $\Sl_2(\R)$ acts on the right on $\H_2$: if $M\in \Sl_2(\R)$ and $M^{-1}=\begin{pmatrix}  a & b \\ c & d  \end{pmatrix}$ then $$\displaystyle{z\cdot M:=M^{-1}z=\frac{az+b}{cz+d}}$$

\subsection{The upper half-plane $\H_2$ as a parameter space for positive definite quadratics.} Let $$Q(X,Z)=aX^2-2bXZ+cZ^2$$ be a binary quadratic form with real coefficients and homogeneous variables $[X,Z] \in \P^1\R$. Let $\Delta=ac-b^2$ be its discriminant. Then $$Q(X,Z)=a[X-(b/a)Z]^2+(\Delta/a)Z^2.$$ For both $\Delta>0$ and  $a>0$, $Q(X,Z)$ is always positive (note that $(X,Z)\neq (0,0)$ since $[X,Z] \in \R\P^1$). Such a quadratic form $Q$ is called \textit{positive definite}.  It has two complex roots $[\omega,1],[\bar{\omega},1]$ where $\omega=b/a+(\sqrt {\Delta}/a) {\bf i}\in \H_2$. Let $V^+_{2,\R}$ be the space of positive definite real quadratic forms. To each $Q(X,Z)\in V^+_{2,\R}$, we associate the complex number $\omega$ in $\H_2$.  

\begin{defn}The map $$\z:V^+_{2,\R}\rightarrow \H_2$$ which sends a positive definite quadratic to its root in $\H_2$ is called  \underline{the zero map}. 
\end{defn}

The hyperbolic plane $\H_2$ is a parameter space for positive definite quadratic forms (up to a constant factor) via the inverse 
\[ \z^{-1}(\omega)=Q_{\omega}:=(X-\omega Z)(X-\bar{\omega}Z).\]
 The group $\Sl_2(\R)$ acts on $V^+_{2,\R}$ via the linear change of variables: for a matrix $M= \begin{pmatrix}  a & b \\ c & d  \end{pmatrix}$, $$(M\cdot Q)(X,Z)=Q^M(X,Z):=Q(aX+bZ,cX+dZ).$$ Note that the $\Sl_2(\R)$ action does not change the discriminant. One can easily verify the following 

\begin{prop}
The zero map $\z: V^+_{2,\R}\rightarrow \H_2$ is $\Sl_2(\R)$-equivariant, i.e. $$\z(M\cdot Q)=M^{-1}\z(Q).$$
\end{prop}

\noindent When $\Delta=0$, the quadratic form $Q(X,Z)=a[X-(b/a)Z]^2$ has a real, double root $[b/a,1]$. If $a$ is a real number, we let $Q_a=(X-aZ)^2$ be the quadratic with a double root at $[a,1]$. We also let $Q_{\infty}=Z^2$ be the quadratic form with a double root at $\infty$. It has thus been established that the boundary $\R\P^1=\R \cup \infty$ of $\H_2$ parametrizes quadratic forms (up to a constant factor) with discriminant $\Delta=0$.

To recap: the hyperbolic plane $\H_2$ parametrizes binary quadratic forms with discriminant $\Delta>0$ and $a>0$, while its boundary parametrizes those with discriminant $\Delta=0$.

It has been claimed and used in \cite{julia} and \cite{SC} that this parametrization is not just a bijection between sets; the hyperbolic geometry of $\H_2$ represents faithfully the algebra of quadratic forms. This was probably known even before. In any case, here is the appropriate statement and a proof of it.
\begin{prop}
Let $\overline{\H}_2=\H_2\cup \partial \H_2=\H_2\cup\R\P^1$ and $\omega_1,\omega_2\in \overline{\H}_2$. The quadratics of the form $$s Q_{\omega_1}+t Q_{\omega_2},s\geq 0,t\geq 0, s+t=1$$ parametrize the hyperbolic segment that joins $\omega_1$ and $\omega_2$. 
\end{prop}
\begin{proof}
We will show only the case when the hyperbolic segment is part of a semicircle. The vertical geodesic is similar. Let $a<b$ be two real numbers such that $A(a,0), B(b,0)$ are the ideal points of the geodesic $C_{a,b}$ that passes through $\omega_1,\omega_2$. We first show that $C_{a,b}$ parametrizes quadratics of the form $$\lambda Q_a+\mu Q_b,\lambda \geq 0,\mu \geq 0, \lambda+\mu=1,$$ i.e. $\z(\lambda Q_a+\mu Q_b)\in C_{a,b}$. The center of $C_{a,b}$ is on the real axis at $\displaystyle{\frac{a+b}2}$ and its radius is $\displaystyle{\frac{b-a}2}$. Let $\lambda \geq 0,\mu \geq 0, \lambda+\mu=1$. Then $$\lambda Q_a+\mu Q_b=\lambda (x-a)^2+\mu (x-b)^2=x^2-2(\lambda a+\mu b)x+\lambda a^2+\mu b^2.$$ The root of $\lambda Q_a+\mu Q_b$ in $\H_2$ is $$(\lambda a+\mu b)+{\bf i}(b-a)\sqrt{\lambda\mu},$$ and its distance from $((a+b)/2,0)$ is easily computed to be $(b-a)/2$.

The proposition now follows easily. Let \[Q_{\omega_1}=\lambda_1Q_a+\mu_1Q_b~ \text{and}~Q_{\omega_2}=\lambda_2Q_a+\mu_2Q_b~\text{with}~\lambda_i+\mu_i=1,~\text{for}~i=1,2.\] Then, for $s\geq 0,t \geq 0,s+t=1$ we have \[s Q_{\omega_1}+t Q_{\omega_2}=(s\lambda_1+t\lambda_2)Q_a+(s\mu_1+t\mu_2)Q_b,~\text{with}~(s\lambda_1+t\lambda_2)+(s\mu_1+t\mu_2)=1,\] hence $\z(s Q_{\omega_1}+t Q_{\omega_2})\in C_{a,b}$. It is obvious that  $\z(s Q_{\omega_1}+t Q_{\omega_2})$ lives in fact in the hyperbolic segment that joins $\omega_1$ and $\omega_2$.
\end{proof}

This proposition can be generalized by induction as follows. 

\begin{prop}\label{convex hull}
Let $\omega_1,\omega_2,...,\omega_n\in \overline{\H}_2$ such that for all $i$, $\omega_i$ is not in the hyperbolic convex hull of $\omega_1,\omega_2,...,\omega_{i-1}$. Then the convex hull of  $\omega_1,\omega_2, \dots ,\omega_n$ parametrizes the linear combinations $\sum_{i=1}^n\lambda_iQ_{\omega_i}$ with $\lambda_i\geq 0$ and $\sum_{i=1}^n\lambda_i=1$.
\end{prop}

\proof
We proceed by induction. For $n=2$ the statement is true due to the previous proposition. Consider $\sum_{i=1}^n\lambda_iQ_{\omega_i}$ with $\lambda_i\geq 0$ and $\sum_{i=1}^n\lambda_i=1$. Then \[\sum_{i=1}^n\lambda_iQ_{\omega_i}=\left(\sum_{i=1}^{n-1}\lambda_i\right )\sum_{i=1}^{n-1}\left(\frac{\lambda_i}{\sum_{i=1}^{n-1}\lambda_i}\right)Q_{\omega_i}+\lambda_nQ_{\omega_n}.\] By induction hypothesis, there exists $\omega_0$ in the convex hull of $\omega_1,\omega_2,...,\omega_{n-1}$ such that \[\sum_{i=1}^{n-1}\left(\frac{\lambda_i}{\sum_{i=1}^{n-1}\lambda_i}\right)Q_{\omega_i}=Q_{\omega_0}\] It follows that \[\sum_{i=1}^n\lambda_iQ_{\omega_i}=\left(\sum_{i=0}^{n-1}\lambda_i\right)Q_{\omega_0}+\lambda_nQ_{\omega_n}\] represents a point $\omega$ in the hyperbolic segment that joins $\omega_0$ and $\omega_n$. Clearly $\omega$ is also in the convex hull of $\alpha_1,\alpha_2,...,\alpha_n$.

\qed

\subsection{The hyperbolic three dimensional space $\H_3$} As a set, $\H_3=\C \times \R^+$. Points of $\H_3$ will be written in the form $z+t{\bf j}$ where $z\in \C$ and $t>0$.  The equation $t=0$ represents the floor $\C$ of $\H_3$. The hyperbolic space $\H_3$  is foliated via horospheres %
\[H_t:=\{z+t{\bf j}:~z\in \C\}\] 
which are centered at $\infty$ and indexed by the height $t$ above $\partial \H_3=\C\P^1$. The algebra of $\H_3$ is not commutative. The following identities are essential to computations: \[{\bf j}^2=-1,~{\bf i}{\bf j}=-{\bf j}{\bf i},~{\bf j}z=\bar{z}{\bf j}~\text{(see the lemma below for a proof of this)}.\] The notion of complex modulus extends to $\H_3$:  $|z+t{\bf j}|=|z|^2+t^2$. There is a natural isometrical inclusion map $\H_2\rightarrow \H_3$ via $x+{\bf i}t\rightarrow x+{\bf j}t$, the upper half-plane $\H_2$ thus, sits as a vertical cross-section inside $\H_3$. The invariant elements of $\H_3$ under the partial conjugation \[z+{\bf j}t \mapsto \bar z+{\bf j}t\] are precisely the elements of $\H_2$. The hyperbolic metric of $\H_3$ is \[ds^2=\frac{|dz|^2+dt^2}{t^2}.\] The geodesics are either semicircles centered on the floor $\C$ and perpendicular to $\C$, or rays $\{z_0+{\bf j}t\}$ perpendicular to $\C$. 

\begin{figure}[htbp] 
   \centering
   \includegraphics[width=3.5in]{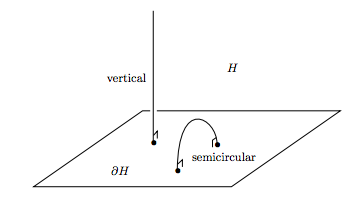} 
   \caption{Geodesics in upper half-space $H$}
   \label{fig5}
\end{figure}

For $\omega=z+t{\bf j}\in \H_3$ and $w+0{\bf j}\in \C$ on the floor, define $$d_H(\omega,w):=\frac{|z-w|^2+t^2}{y}.$$ The following proposition and its proof are straightforward generalizations from $\H_2$.

\begin{prop}
If one of the ideal points of the geodesic through $\omega_1,\omega_2$ is at $w$, then $$d_H(\omega_1,\omega_2)=|d_H(\omega_1,w)-d_H(\omega_2,w)|.$$ 
\end{prop}

There is a right action of $\Sl_2(\C)$ on $\H_3$. If $M\in \Sl_2(\C)$ and $M^{-1}= \begin{pmatrix}  a & b \\ c & d  \end{pmatrix}$, its action is described as follows \[(z+{\bf j}t)\cdot M=M^{-1}(z+{\bf j}t)=[a(z+{\bf j}t)+b][(c(z+{\bf j}t)+d]^{-1},\] where the inverse indicates the right inverse in the non commutative structure of $\H_3$.  Note that for $t=0$ we get the standard $\Sl_2(\C)$-action on the boundary $\C\P^1$ of $\H_3$.

\begin{lem} The action of $\Sl_2(\C)$ on $\H_3$ can be written in the form
\[(z+{\bf j}t)\cdot M=\frac{(az+b)\overline{(cz+d)}+a\bar ct^2+{\bf j}t}{|cz+d|^2+|c|^2t^2}.\]
\end{lem}

\begin{proof}
First, with $z=x+y{\bf i}$ we have \[{\bf j}\bar z={\bf j}(x-y{\bf i})=x{\bf j}-y{\bf j}{\bf i}=x{\bf j}+y{\bf i}{\bf j}=(x+y{\bf i}){\bf j}=z{\bf j}.\] Using this identity, it is straightforward to show that \[\left[c(z+t{\bf j})+d\right]\left[(\bar z-t{\bf j})\bar c+\bar d\right]=|cz+d|^2+t^2|c|^2.\] Real numbers commute with both ${\bf i}$ and ${\bf j}$ in $\H_3$, hence they have a well-defined inverse. We obtain the right inverse as follows:  \[\left[c(z+t{\bf j})+d\right]^{-1}=\frac{(\bar z-t{\bf j})c+\bar d}{|cz+d|^2+t^2|c|^2}.\] The lemma follows from the straightforward calculation \[[a(z+{\bf j}t)+b][(\bar z-t{\bf j})+\bar d]=(az+b)(\overline{cz+d})+a\bar c t^2+t{\bf j}.\]

\end{proof}

\subsection{The upper half-space $\H_3$ as a parameter space for positive definite Hermitian quadratics.} Let $$H(X,Z)=a|X|^2-bX\Bar Z-\bar b\bar XZ+c|Z|^2, a,c\in \R$$ be a Hermitian quadratic form with homogeneous variables $[X,Z] \in \P^1\C$. Notice that the values of $H(X,Z)$ are always real. Let $\Delta=ac-|b|^2$ be its discriminant. Then $$H(X,Z)=a[X-(\bar b/a)Z]^2+(\Delta/a)Z^2,$$ hence $H(X,Z)>0$ for all $(X,Z)$ when $\Delta >0,a>0$. Such a form is called \textit{positive definite.}  Denote the set of all positive definite Hermitian forms by $V^+_{2,\C}$. There is an $\Sl_2(\C)$ action on $V^+_{2,\C}$ similar to the real case. The natural $\Sl_2(\R)$ equivariant inclusion $\psi: V^+_{2,\R}\rightarrow V^+_{2,\C}$ via $$\psi(aX^2-2bXZ+cZ^2)=a|X|^2-bX\bar Z-\bar b\bar XZ+c|Z|^2,$$ gives rise to an extension of the zero map.

\begin{defn}
The zero map $\z: V^+_{2,\C}\rightarrow \H_3$ is defined via
\begin{equation} \label{eq: hermitian zero map}
\z(a|X|^2-bX\bar Z-\bar b\bar XZ+c|Z|^2)=\frac{\bar b}{a}+{\bf j}\frac{\sqrt \Delta}{a}
\end{equation}
\end{defn}

\begin{prop} \label{complex invariant}
The map $\z$ is $\Sl_2(\C)$ equivariant.
\end{prop}

\begin{proof}
The generators of $\Sl_2(\C)$ are matrices of the form $ \begin{pmatrix}  0 & a \\ 0 & 1  \end{pmatrix},~ \text{for}~a\in \C$ and $\begin{pmatrix}  0 & -1 \\ 1 & 0  \end{pmatrix}$. It is easy to show that for any generator matrix $M$: $$\z(H^M)=M^{-1}\z(H).$$
\end{proof}
 
The hyperbolic space $\H_3$ is a parameter space for positive definite ($\Delta>0,a>0$) Hermitian forms via the inverse map $$\z^{-1}(\omega)=\z^{-1}(z+{\bf j}t)= |X|^2-\bar{z}\bar XZ-zX\bar Z+(|z|^2+t^2)|Z|^2=H_{\omega}.$$ The boundary $\C\P^1=\C\cup{\infty}$ of $\H_3$ is a parameter space for the decomposable ($\Delta=0$) Hermitian forms \[H_{\beta}=(X-\bar \beta Z)(\bar X-\beta \bar Z)~\text{for}~\beta\in \C,~H_{\infty}=|Z|^2,\] Just as in the case of the upper half-plane $\H_2$, we have the following proposition:

\begin{prop} \label{convex hull H3}
Let $\overline{\H}_3=\H_3\cup \partial \H_3=\H_3\cup \C\P^1$. The hyperbolic convex hull of $\omega_1,\omega_2,...,\omega_n\in \overline{\H}_3$ parametrizes Hermitian forms $\sum_{i=1}^n \lambda_i H_{\omega_i}$ with $\lambda_i \geq 0$ for $i=1,2,...,n$ and $\sum_{i=1}^n \lambda_i=1$. 
\end{prop}

The equivariant connection between the geometry of hyperbolic spaces and the algebra of positive definite forms, which extends to the boundary as well, can be expressed in the following equivariant commutative diagram:

\[
\xymatrix{
 V^+_{2,\R}   \ar@{->}[d]  \ar@{->}[r]^\z      & \H_2  \ar@{->}[d]^{}  \\
 V^+_{2,\C}        \ar@{->}[r]^\z                   & \H_3 \\
}
\]
Next, we will see how to use the equivariance of the zero map to construct a reduction method.  
\section{Reduction of binary forms via the Julia quadratic} 

In this section we summarize the reduction of binary forms via the zero map obtained in \cite{julia} and \cite{SC}. We will focus especially on the geometric features of the theory which are of particular interest to us.

Let $V_n(\C)$ denote the space of complex binary forms of degree $n$. If $F\in V_n(\C)$ then
 \[ F(X,Z)=a_0\prod_{i=1}^n(X-\alpha_iZ) \]
for some complex numbers $\alpha_j$ and $a_0\neq 0$. For $t_1,t_2,...,t_n\geq 0$ define \[Q_F(t_1,t_2,...,t_n)=\sum_{i=1}^{n}t_i|X-\alpha_iZ|^2=\sum_{i=1}^{n}t_iH_{\alpha_i}(X,Z).\]

From Proposition \ref{convex hull H3} above, the positive definite Hermitian forms $Q_F(t_1,t_2,...,t_n)$ parametrize the hyperbolic convex hull of $\alpha_1,\alpha_2,...,\alpha_n\in \overline{\H}_3$. Let $(t_1^0,t_2^0,...,t_n^0)$ be the values that minimize \[\theta_0:=\frac{a_0^2(\text{disc}(Q_F))^{n/2}}{n^nt_1t_2...t_n}.\] 

\begin{defn}
The form $\J_F:=Q_F(t_1^0,t_2^0,...,t_n^0)\in V^+_{2,\C}$ is called the \underline{Julia quadratic} of $F$. The zero map extends to $\z: V_n(\C)\rightarrow \H_3$ via $\z(F)=\z(J_F)\in \H_3$. The form $F$ is called \underline{reduced} if $\z(F)$ is in the fundamental domain $\mathcal F$ of $\Sl_2(\C)$.
\end{defn}

To reduce a real binary form $F(X,Z)$ we first compute its zero map value $\z(F)$ in $\H_2$. If $\z(F)$ is in the fundamental domain $\F$ of $\Sl_2(\R)$ then $F(X,Z)$ is already reduced. If not, let $M\in \Sl_2(\R)$ such that  $\z(F)\cdot M\in \F$. The form $F(X,Z)$ reduces to $F^{M}(X,Z)$ which is expected to have smaller coefficients. Similar procedure holds in $\H_3$ for complex binary forms. 

In \cite{SC}, the authors provide a geometric description of the zero map. The roots $\alpha_i,~i=1,2,...,n$ of $F(X,Z)$ are placed in the floor $t=0$ of $\H_3$. Recall the hyperbolic distance between a point $w=z+{\bf j}t\in \H_3$ and $\omega \in \C$ in the boundary floor: $$d_H(w,\omega)=\ln \frac {|z-\omega|^2+t^2}{t}.$$

\begin{prop}
(Proposition $5.3$ in \cite{SC}) The zero map value $\z(F)$ is the unique point $w_0\in \H_3$ that minimizes the sum of distances $$\tilde{F}(w):=\displaystyle{\sum_{i=1}^n d_H(w,\alpha_i)}.$$
\end{prop}

We emphasize that the minimized sum of the hyperbolic distances is not $\Sl_2(\C)$-invariant but its sum with $2\ln a_0$ is.  Here $M\in \Sl_2(\C)$ acts by the linear change of variables on $F(X,Z)$ and acts on the right on $w_0$. 

Another equivalent, geometric description of the zero map is given by the following statement:

\begin{cor}\label{tangent vectors}
(Corollary $5.4$ in \cite{SC}) The zero map value $\z(F)$ is the unique point $w_0\in \H_3$ such that the unit tangent vectors at $w_0$ along the geodesics to the roots $\alpha_i$ add up to zero.
\end{cor}

As mentioned above, this minimizing solution $w_0$ is $\Sl_2(\C)$-invariant. Furthermore, when $F(X,Z)$ has real coefficients, $w_0$ is also invariant with respect to the partial conjugation $w_0=z_0+t_0{\bf j}\mapsto \bar z_0+t_0{\bf j}$. Hence, $z_0$ is real number, i.e. $w_0\in \H_2$.

\section{The reduction of real forms via the hyperbolic center of mass}

In this section we introduce an alternative zero map for binary forms with real coefficients and no real roots. It is based on the notion of \textit{hyperbolic center of mass} in hyperbolic spaces. We focus in $\H_2$ which is the case of interest for us, but the general case is straightforward. Our treatment follows closely that of \cite{gal}.
 
\subsection{The center of mass via the hyperboloid model of the hyperbolic plane}  Met $M$ be the Minkowski pairing in $\R^3$: for ${\bf x}=(x_1,x_2,x_3),{\bf y}=(y_1,y_2,y_3)$
\[ M({\bf x}, {\bf y})=-x_1y_1-x_2y_2+x_3y_3.\]
Denote the corresponding norm $||{\bf x}||^2=M({\bf x},{\bf x})=-x_1^2-x_2^2+x_3^2$. Let $\H$ be the upper sheet of the hyperboloid \[\H:=\{{\bf x}~:~||{\bf x}||=1,~x_3>0\}.\] Its equation is $-x_1^2-x_2^2+x_3^2=1$ and its metric is given by $ds^2=dx_1^2+dx_2^2-dx_3^2$. If ${\bf x},{\bf y}\in \H$, the hyperbolic distance $d_H({\bf x},{\bf y})$ in this model can be found via \[\cosh d_H({\bf x},{\bf y})=M({\bf x},{\bf y}).\]

\begin{defn}
Let ${\bf x}_j\in \H,~j=1,2,...,r$. Their \underline{center of mass} is defined as $${\mathcal C}={\mathcal C}_{\H}(x_1,x_2,...,x_r):=\displaystyle{\frac{\sum_{j=1}^r{\bf x}_j}{||\sum_{j=1}^r{\bf x}_j||}}.$$ 
\end{defn}
Notice that $$\sum_{i=1}^r \cosh (d_H({\mathcal C},{\bf x}_i))=\sum_{i=1}^r M({\mathcal C},{\bf x}_i)=M({\mathcal C},\sum_{i=1}^r{\bf x}_i)=||\sum_{i=1}^r{\bf x}_i||M({\mathcal C},{\mathcal C})=||\sum_{i=1}^r{\bf x}_i||$$

\begin{prop}
The center of mass ${\mathcal C}_{\H}(x_1,x_2,...,x_r)$ is $\Sl_2(\R)$ invariant. It is the unique point ${\bf x}\in \H$ that minimizes $\sum_{j=1}^r \cosh (d_H({\bf x},{\bf x}_j))$. 
\end{prop}

\begin{proof}
Recall that $\Sl_2(\R)$ action on $\H$ preserves hyperbolic distances, hence the center of mass is $\Sl_2(\R)$ invariant. The proof of the second part follows easily by solving the minimizing problem $$\text{minimize}~M({\bf x},\sum_{j=1}^r{\bf x}_j),~~\text{subject to}~{\bf x}\in \H$$  using the Lagrange Multipliers method and the inequality $||\sum_{j=1}^r{\bf x}_j||>1$.

\end{proof}

We use the minimizing property to transfer the notion of center of mass in $\H_2$. There is an isometry $\H_2\rightarrow \H$ given by 
 \[ u+{\bf i}v \rightarrow \left(   \frac{1-u^2-v^2}{2u}, \frac{u}{v},\frac{1+u^2+v^2}{2v} \right).\] The following identity holds in $\H_2$:
$$\cosh d_H(z_1,z_2)=1+\frac{|z_1-z_2|^2}{2y_1y_2}$$ 
for $z_1=x_1+{\bf i}y_1\in \H_2,~z_2=x_2+{\bf i}y_2\in \H_2$.  It follows that if $\alpha_j=x_j+{\bf i}y_j\in \H_2,~j=1,2,...,n$, their center of mass is the complex number $t+{\bf i}u\in \H_2$ such that $$\sum_{j=1}^n \left[1+\frac{(t-x_j)^2+(u-y_j)^2}{2uy_j}\right ]$$ is minimal. By excluding the constant summands, we obtain the following:

\begin{defn} The hyperbolic center of mass ${\mathcal C}_{\H}(\alpha_1,\alpha_2,...,\alpha_n)$ of the collection $\{\alpha_j\in \H_2~ | j=1,2,...,n\}$ is the unique point $t+{\bf i}u\in \H_2$ that minimizes $$\sum_{j=1}^n \frac{(t-x_j)^2+(u-y_j)^2}{uy_j}.$$ 
\end{defn}

Setting the partials equal to zero, we obtain a system of equations for the center of mass ${\mathcal C}_{\H}(\alpha_1,\alpha_2,...,\alpha_n)=t+{\bf i}u\in \H_2$:

\begin{equation}\label{center of mass system}
\left\{
\begin{split}
& \sum_{j=1}^n \frac{t-x_j}{y_j} = 0\\
&  \sum_{j=1}^n \frac{u^2-(t^2-2x_jt+x_j^2+y_j^2)}{y_j} = 0
\end{split}
\right.
\end{equation}
With substitutions
\begin{equation}\label{eq_q}
q_j =q_j(t):= t^2- 2 x_j t + x_j^2 +y_j^2=Q_{\alpha_j} (t,1), 
\end{equation}
 the solution of the above system is given by 
\begin{equation}\label{centroid}
\begin{split}
 \displaystyle{t} & \displaystyle{= \frac { \sum_{i=1}^n   y_1 y_2 \cdots y_{i-1} x_i y_{i+1} \cdots y_n }   { \sum_{i=1}^n   y_1 y_2 \cdots y_{i-1} \,  y_{i+1} \cdots y_n  }}  \\
\displaystyle{u^2}  & =  \displaystyle{\frac { \sum_{i=1}^n   y_1 y_2 \cdots y_{i-1} q_i y_{i+1} \cdots y_n }   { \sum_{i=1}^n   y_1 y_2 \cdots y_{i-1} \,  y_{i+1} \cdots y_n  }}  \\
\end{split}
\end{equation}
%
%
%
%
%

Solutions to the system of the equations in Eq.~\eqref{centroid} are easy to describe. Let $\psi : \R^n \times \R^n   \mapsto \R$, be defined by
\begin{equation}\label{map}
\begin{split}
\psi\left(     \left( x_1, \dots , x_n \right), \left( y_1, \dots , y_n \right) \right)  &= \frac { \sum_{i=1}^n   y_1 y_2 \cdots y_{i-1} x_i y_{i+1} \cdots y_n }   { \sum_{i=1}^n   y_1 y_2 \cdots y_{i-1} \,  y_{i+1} \cdots y_n  }   \\
\end{split}
\end{equation} 
Let ${\bf x} = ( x_1, \dots , x_n )$,  ${\bf y} = ( y_1, \dots , y_n )$.
%
Now we have
\begin{equation} 
t= \psi ({\bf x}, {\bf y}  ), \qquad u= \sqrt{ \psi (  {\bf q}, {\bf y}  )}~ \text{where}~ {\bf q} = ( q_1(t), \dots , q_n(t) ) 
\end{equation}

\begin{rem}
The function $\psi$ has symmetries and is a combination of $x_i$'s with positive weights that add to one. Weights depend only on $y_i$'s. It is probably a well-known and standard function in areas where symmetries and group actions are relevant.
\end{rem}

We now introduce an alternative reduction theory based on the notion of the hyperbolic center of mass. Let $V^+_{2n,\R}(0,n)$ denote binary forms of degree $2n$ with real coefficients and no real roots. Every $F(X,Z)\in V^+_{2n,\R}(0,n)$ can be factored 
\[F(X,Z)=\prod_{i=1}^n Q_{\alpha_i}(X,Z)\] where 
\[Q_{\alpha_i}(X,Z)=(X-\alpha_iZ)(X-\overline{\alpha_i}Z)\]

\begin{defn}
The hyperbolic center zero map $\z_{\mathcal C}: V^+_{2n,\R}(0,n)\rightarrow \H_2$ is defined via 
\[\z_{\mathcal C}(F):=\mathcal C={\mathcal C}_{\H}(\alpha_1,\alpha_2,...,\alpha_n).\] 
The form
\[\J^{\mathcal C}_F= (X-{\mathcal C}Z)(X-\overline{\mathcal C}Z)\] 
is called the hyperbolic center quadratic of $F$.
\end{defn}

The reduction theory based on the hyperbolic center of mass proceeds as before. Let $F(X,Z)$ be a real binary form with no real roots. If $\z^{\mathcal C}(F)\in \F$ then $F$ is reduced. Otherwise, let $M\in \Sl_2(\R)$ such that $\z^{\mathcal C}(F)\cdot M \in \F$. The form $F$ reduces to $F^M(X,Z)$.

Here is a comparison between the reduction of \cite{julia} \cite{SC} and the one via the hyperbolic center.

\begin{exa}
Let $F(X,Z)$ be the binary sextic with roots $\a_1 = 2+3i$, $\a_2 = 6+4i$, $\a_3 =4+7i$ and their conjugates. Then 
\[ F(X, Z) = (X^2-4 X + 13) (X^2 - 12 X + 52) (X^2 - 8 X + 65). \]

Consider the genus 2 curve
\[ y^2 = X^6 - 24 X^5 + 306 X^4 - 2308 X^3 +     10933 X^2 - 29068 X + 43940 \]
with height $\mathfrak h = 43940$.
Reducing it via \cite{SC} yields a curve $C^\prime$ with equation
\[ y^2 + (X^3 + X)y =  16 X^4 + 7 X^3 + 273 X^2 +    343 X + 3185  \]
which is isomorphic to 
\[
\begin{split}
 Y^2 & = (X^3+X)^2 + 4 \, \left(  16 X^4 + 7 X^3 + 273 X^2 +    343 X + 3185  \right) \\
 & X^6 + 66X^4 + 28 X^3 + 1093X^2 + 1372 X + 12740.   \\
\end{split}
 \]
This last curve has height $\mathfrak h = 12 740$, which is smaller than the original height. 

The reduction  via the hyperbolic center of mass is as follows.  The zero map $\z_{\mathcal C}(F)$ is 
\[ \z_{\mathcal C}(F) = \frac {230} {61} + i \; \frac {14} {61} \, \sqrt{2\cdot 3 \cdot 71} \approx 3.77 + i\,  4.73 \]
To bring this point to the fundamental domain we have to shift it to the left by 4 units.  Hence, we must compute 
\[  f(X+4) = X^6 + 66X^4 + 28 X^3 + 1093X^2 + 1372 X + 12740.\]
which has height $\mathfrak h = 12740$, the same as in the Julia case. 
\end{exa}

We generalize the case of totally complex sextics with the following lemma.   

\begin{lem}
Let $F(X, Z) \in \Z[X, Z] $ be a totally complex sextic factored over $\R$ as 
\[ F(X,Z)= (X^2 + a_1 X Z + b_1 Z^2) (Z^2 + a_2 X Z + b_2 Z^2) (X^2 + a_3 X Z + b_3 Z^2) \]
and denote by $d_j = \sqrt{4b_j-a_j^2}$, for $j=1, 2, 3$, ${\bf d}=(d_1, d_2, d_3)$, and ${\bf a} = (a_1, a_2, a_3)$. 
The hyperbolic center zero map of $F$ is given by 
\begin{small}
\[
\begin{split}
  t & = -\frac{1}{2}\psi({\bf a}, {\bf d}) \\
   u^2 &  = \psi({\bf b},{\bf d})-\frac{1}{4}\psi^2({\bf a},{\bf d})
\end{split}
\]
\end{small}
The hyperbolic center quadratic $\J^{\mathcal C}_F$ is defined over $Q(\sqrt{d_1}, \sqrt{d_2}, \sqrt{d_3})$. 
\end{lem}

\proof Let $\alpha_j=x_j+{\bf i}y_j,~i=1,2,3$ be the roots of $F$.  Since $d_j = \sqrt{4b_j-a_j^2}$, then $x_j = - \frac 1 2 a_j$ and $y_j = \frac 1 2 d_j$. The formulas for $t$ and $u^2$ in terms of the roots $\alpha_j$ are
\[
t= \frac {y_1 y_2 x_3 + y_1 y_3 x_2 + y_2 y_3 x_1 } { y_1 y_2 + y_1 y_3 + y_2 y_3 },    \qquad u^2 = \frac {y_1 y_2 q_3 + y_1 y_3 q_2 + y_2 y_3 q_1 } { y_1 y_2 + y_1 y_3 + y_2 y_3 } 
\]
Substituting $x_j,y_j,~j=1,2,3$ yields immediately the result for $t$. To obtain the expression for $u^2$,  we substitute \[y_j=\frac{d_j}{2}~\text{and}~q_j=q_j(t)=t^2+a_jt+b_j\] in the formula for $u^2$

\[u^2 = \frac {y_2 y_3 }{ y_1 y_2 + y_1 y_3 + y_2 y_3}q_1 + +\frac{y_1 y_3 }{ y_1 y_2 + y_1 y_3 + y_2 y_3}q_2 + \frac{y_1 y_2 }{ y_1 y_2 + y_1 y_3 + y_2 y_3}q_3\]

Simple algebra yields \[u^2=t^2+t\psi({\bf a},{\bf d})+\psi({\bf b},{\bf d})\] and the result follows by substituting $t=-\frac{1}{2}\psi({\bf a}, {\bf d}) $.
\qed

Note that this lemma is an improvement compared to results obtained in \cite{beshaj} and \cite{SC}, where the zero map of a sextix is given in terms of its roots.  It is straightforward to generalize and prove these results to any degree. 

\begin{prop}
Let $F(X, Z)$ be a totally complex form factored over $\R$ as below
\[ F(X, Z)= \prod_{i=1}^r (X^2 + a_i X Z  + b_i Z^2) \]
Denote by $d_i = \sqrt{4b_i-a_i^2}$, for $i=1, \dots ,  r$ the discriminants for each factor of $F(X,Z)$,  ${\bf d}=(d_1, \dots , d_r)$, and ${\bf a} = (a_1, \dots , a_r)$.     
Then, the image $\z(F)$ of the zero map is given by
\begin{equation}
\begin{split}
  t &= \, -\frac{1}{2}\psi({\bf a}, {\bf d}) \\
   u^2 &  = \psi({\bf b},{\bf d})-\frac{1}{4}\psi^2({\bf a},{\bf d})    
\end{split}
\end{equation}
The hyperbolic center quadratic $\J^{\mathcal C}_F$ is defined over $Q(\sqrt{d_1}, \sqrt{d_2}, \sqrt{d_3})$. . 
\end{prop}

\begin{rem} Substituting $q_j=(t-x_j)^2+y_j^2$ in the formula for $u^2$, we obtain the following alternative presentations for $t$ and $u^2$. Here interested readers can see various symmetries with respect to the coefficients of the quadratic factors.
\begin{equation}
\begin{split}
  t & =  - \frac 1 {2\mathfrak s_{n-1}}  \,     \sum_{i=1 }^r d_1 \cdots  d_{i-1}a_id_{i+1}\cdots d_r  \, ,~ \\
   u^2 &  = \frac  1   { 4 \mathfrak s_{n-1}^2}     \, \prod_{i=1}^r d_i      \,  
   \left( \mathfrak s_{n-1} \, \sum_{i=1}^r d_i  +      \sum_{i}^r \, d_1 \cdots \hat{d_i}\cdots \hat{d_j}\cdots d_r  \left( a_i - a_j  \right)^2    \right)    \\
\mathfrak s_{n-1} & = \sum_{i=1 }^r d_1 \cdots  d_{i-1}\hat{d_i}d_{i+1}\cdots d_r\\   
\end{split}
\end{equation}
where $\hat{x}$ denotes a missing $x$.
\end{rem}

It would be interesting to express $\z(F)$ in terms of invariants of $F$ or symmetries of the roots of $F$, and as a more overarching goal, to incorporate the real roots of the binary form $F$ in this approach. We will continue to explore these issues.


\bibliographystyle{amsalpha} 

\bibliography{ref}{}

\end{document}